\documentclass{amsart}
\usepackage[utf8]{inputenc}
\usepackage{amsfonts}
\usepackage{amsmath}
\usepackage{amssymb}
\usepackage{amsthm}
\usepackage{hyperref}
\usepackage{xcolor}
\usepackage[shortlabels]{enumitem}
\usepackage[all,cmtip]{xy}
\usepackage{hyperref}
\hypersetup{colorlinks=true,allcolors=blue}

\newcommand{\nc}[1]{\newcommand{#1}}
\nc{\on}[1]{\operatorname{#1}}
\nc{\QQ}[0]{\mathbb Q}
\nc{\End}[0]{\on{End}}
\nc{\Br}[0]{\on{Br}}

\newtheorem{theorem}{Theorem}[section]
\newtheorem{proposition}[theorem]{Proposition}
\newtheorem{lemma}[theorem]{Lemma}

\nc{\ord}[0]{\on{ord}}

\title{Abelian varieties that split modulo all but finitely many primes}
\author{Enric Florit}
\address{Departament de Matemàtiques i Informàtica, Universitat de Barcelona, Gran via de les Corts Catalanes, 585, 08007 Barcelona, Spain}
\email{enricflorit@ub.edu}
\date{April 11, 2024}

\begin{document}

\begin{abstract}
Let $A$ be a simple abelian variety over a number field $k$ such that $\End(A)$ is noncommutative. We show that $A$ splits modulo all but finitely many primes of $k$. We prove this by considering the subalgebras of $\End(A_{\mathfrak p})\otimes\QQ$ which have prime Schur index. Our main tools are Tate's characterization of endomorphism algebras of abelian varieties over finite fields, and a Theorem of Chia-Fu Yu on embeddings of simple algebras.
\end{abstract}

\maketitle

\section{Introduction}

Let $k$ be a number field and let $A$ be a simple abelian variety over $k$. Let $\End^0(A):= \End(A)\otimes \QQ$ be the algebra of endomorphisms of $A$ defined over $k$. For a prime $\mathfrak p$ of good reduction for $A$, we denote by $A_{\mathfrak p}$ the reduction of $A$ modulo $\mathfrak p$. We say $A_{\mathfrak p}$ splits if it is isogenous (over the residue field of $\mathfrak p$) to a product of abelian varieties of lower dimension. In this note we show the following.

\begin{theorem}\label{theorem:main}
	Suppose $\End^0(A)$ is noncommutative. Then, for every prime $\mathfrak p$ of $k$ of good reduction for $A$ coprime to all primes of ramification of $\End^0(A)$, the reduction $A_{\mathfrak p}$ splits. In particular, $A_{\mathfrak p}$ splits for all but finitely many primes $\mathfrak p$.
\end{theorem}

This result generalizes the analogous theorem for abelian surfaces with QM by Morita and Yoshida \cite{morita70, yoshida73}. The proof of Theorem~\ref{theorem:main} is guided by 2-dimensional case (see Proposition~\ref{proposition:dim-2} below). There, the isomorphism
\begin{equation}\label{eq:1}
	\End^0(A_{\mathfrak p})\simeq \End^0(A)\otimes\QQ(\pi)
\end{equation}
holds, with $\pi$ the Frobenius endomorphism of $A_{\mathfrak p}$. The characterization of the ramified places of $\End^0(A_{\mathfrak p})$ by Tate gives the ramification condition on $\End^0(A)$. 

To generalize to higher dimension, we need to find a suitable replacement of the isomorphism in \eqref{eq:1}, which does not hold in general.\footnote{However, this isomorphism does hold at a set of primes of density 1 under the Mumford-Tate conjecture. See \cite[Theorem~1.2]{zywina14} for more details.}
Instead, we work with classes in a suitable Brauer group. After extending scalars, we are able to compare the classes $[\End^0(A)]$ and $[\End^0(A_{\mathfrak p})]$, this is enough to make statements about ramification. In order to compare these classes, we study the subalgebras of $\End^0(A_{\mathfrak p})$ that have prime Schur index (recall that the Schur index of a central simple algebra $X$ over a number field $Z$ is the order of the class $[X]$ in the Brauer group $\Br(Z)$). This is the content of Theorem~\ref{theorem:main-theorem-algebras}, which is our main technical result. In short, our work studies the endomorphism subalgebras of simple abelian varieties defined over finite fields.

Some work remains to be done regarding the finite set of primes $\mathfrak p$ such that $A_{\mathfrak p}$ is simple. If $A$ is a surface with QM and $\End^0(A)$ ramifies at the rational prime $p$ below $\mathfrak p$, then $A_{\mathfrak p}$ has $p$-rank 0, so it is supersingular, and in particular it is geometrically split. This argument breaks down if $\dim A>2$ for at least two reasons. First, the $p$-rank can take different values depending on the dimension of $\End^0(A)$ relative to $\dim A$. Second, having $p$-rank 0 does not imply supersingularity in higher dimension \cite[pg.~9]{lioort98}. At the time of writing, the author does not know of any abelian variety $A$ with noncommutative endomorphism algebra such that, for a prime of good reduction $\mathfrak p$, $A_{\mathfrak p}$ is geometrically simple.

Theorem~\ref{theorem:main} was known to hold for primes $\mathfrak p$ of $k$ with prime residue field, which form a set of density 1. This is \cite[Lemma~2.6]{achter09} and \cite[Lemma~2.1]{zywina14}, who in turn use \cite[Theorem~6.1]{waterhouse69}. The proof for these primes uses the well-known formula by Tate that computes the local invariants of $\End^0(A_{\mathfrak p})$ from the Frobenius endomorphism $\pi$, but it does not generalize to other primes $\mathfrak p$. Achter and Zywina also show (conditionally on the Mumford-Tate conjecture) that --after possibly replacing $k$ with a finite extension-- an abelian variety $A$ with commutative $\End(A_{\bar k})$ remains simple modulo all primes $\mathfrak p$ in a set of density 1, as conjectured by Murty and Patankar in \cite{murty08}. This situation seems qualitatively different from ours: in some cases \cite{shankar-tang20}, the (density zero) set of primes $\mathfrak p$ where $A_{\mathfrak p}$ splits is infinite.

The remainder of this document is organized as follows. In Section~\ref{section:embeddings} we first recall a theorem of Yu, which gives us a criterion to work with embeddings of simple algebras. We then show the existence of subalgebras of prime Schur index in Section~\ref{section:prime-subalgebra}. We prove Theorem~\ref{theorem:main-theorem-algebras} in Section~\ref{section:theorem-algebras}, this characterizes all subalgebras of a division algebra having prime Schur index. Finally, we prove Theorem~\ref{theorem:main} in Section~\ref{section:proof-of-main-theorem}.

We refer the reader to \cite{pierce82} for the missing background on central simple algebras (particularly, Chapters 9, 12, 13 and 18).

\subsection*{Notation}

All algebras that appear are finite-dimensional over $\QQ$. In particular, every field is a number field. If $X$ is a simple algebra, $\mathrm{M}_n(X)$ denotes the algebra of $n$-by-$n$ matrices with entries in $X$. Every simple algebra $X$ has an opposite algebra, denoted by $X^{op}$, whose set is the same as $X$ and whose multiplication is reversed.

\subsection*{Acknowledgements}

I wish to thank Francesc Fité and Xavier Guitart for their guidance and comments during this project. I also thank Javier Guillán-Rial for some fruitful discussions on key lemmas. This work was supported by the Spanish Ministry of Universities (FPU20/05059) and by projects PID2019-107297GB-I00, PID2022-137605NB-I00 and 2021 SGR 01468.

\section{Embeddings of simple algebras}\label{section:embeddings}

Let $Q$ be a number field. Let $X$ and $Y$ be simple finite-dimensional $Q$-algebras, possibly with larger centers. Recall that an \emph{embedding of $Q$-algebras} $\iota\colon X\to Y$ is an injective ring homomorphism which is $Q$-linear. In particular, we have $\iota(1)=1$ and $\iota(qx)=q\iota(x)$ for all $q\in Q$ and all $x\in X$. Sometimes we also say $\iota$ is an \emph{embedding} when it is an embedding of $\QQ$-algebras, without any implication for the linearity with respect to a larger field.

Given a simple algebra $X$, by Wedderburn's structure theorem \cite[3.5~Theorem]{pierce82} there exists a division algebra $X'$ and a positive integer $c$ such that $X=\mathrm{M}_c(X')$. We call $c=c(X)$ the \emph{capacity} of $X$.

\begin{theorem}[Chia-Fu Yu]
\label{theorem:yu's-theorem}
    Let $X$ and $Y$ be two simple algebras with centers $Z_X$ and $Z_Y$, respectively. Assume $Z_X\supseteq Z_Y$. Then, there is an embedding of $Z_Y$-algebras of $X$ into $Y$ if and only if $\dim_{Z_Y}X$ divides the capacity of the simple algebra $Y\otimes_{Z_Y} X^{op}\simeq (Y\otimes_{Z_Y}Z_X)\otimes_{Z_X}X^{op}$.
\end{theorem}
\begin{proof}
    This is a particular case of \cite[Theorem~1.2]{yu12}, and is also proved in \cite[Proposition~2.2]{yu13}.
\end{proof}

For a simple algebra $X$ with center $Z_X$, we denote by $\ord_{Z_X}[X]$ the order of $[X]$ in the Brauer group $\Br(Z_X)$. This order is also called the \emph{Schur index} of $X$. The dimension, Schur index and capacity of $X$ are related by the equality
\[
	\dim_{Z_X}X = (c(X)\cdot \ord_{Z_X}[X])^2.
\]
Note that $\dim_{Z_X} X=\ord_{Z_X}[X]^2$ exactly when $X$ is a division algebra.

\subsection{Finding a prime subalgebra}
\label{section:prime-subalgebra}

We apply Theorem~\ref{theorem:yu's-theorem} to find algebras of prime Schur index in division alebras.
 
\begin{lemma}\label{lemma:subalgebra-prime-Schur-index}
    Let $E$ be a division algebra with center $Z$ with Schur index $m$. Let $\ell$ be a prime factor of $m$. Then $E$ contains a field $F$ with $F\supseteq Z$ and a simple algebra $D$ with center $F$ and Schur index $\ell$ such that
    \(
    [E\otimes_{Z} F] = [D]
    \)
    in $\Br(F)$.
\end{lemma}
\begin{proof}
    Because $Z$ is a number field, $E$ is a cyclic algebra \cite[18.8~Theorem]{pierce82}. By definition, this means that there is a maximal field $M\subset E$ such that $M/Z$ is a cyclic Galois extension. The degree of this extension is $[M:Z]=\ord_{Z}[E]=m$. Hence there is a subfield $F\subset M$ containing $Z$ and such that $[M:F]=\ell$. Now let $D$ be the unique division algebra in the Brauer class of $[E\otimes_{Z} F]$. 

    We need to check that $D$ is indeed a subalgebra of $E$. Note that
    \[
    [(E\otimes_{Z} F)\otimes_{F} D^{op}] = [E\otimes_{Z} F] - [D] = [F]
    \]
    in $\Br(F)$, so by counting dimensions we get that 
    $(E\otimes_{Z} F)\otimes_{F} D^{op} \simeq \mathrm{M}_{\ell^2[F:Z]}(F)$. In other words, the capacity of $(E\otimes_{Z} F)\otimes_{F} D^{op}$ is exactly $\ell^2[F:Z]$, and we have $\dim_{Z} D = \ell^2[F:Z]$. Theorem~\ref{theorem:yu's-theorem} with $X=D$ and $Y=E$ implies that $D$ is a subalgebra of $E$.
\end{proof}

\subsection{Embeddings of division algebras}\label{section:theorem-algebras}

In this section we prove our technical result on division algebras using Theorem~\ref{theorem:yu's-theorem}. To do this, it is easier to first perform an extension of scalars.

\begin{lemma}\label{lemma:extension-of-scalars}
	Let $D$ and $B$ division algebras with centers $F$ and $K$, respectively. Suppose we are given an embedding $\iota \colon  D\hookrightarrow B$. Then, the compositum $\tilde F=\iota(F)K$ in $B$ is a field, and $\iota$ extends to an embedding of $K$-algebras
    \[
    \tilde\iota\colon  D\otimes_F \tilde F\hookrightarrow B.
    \]
\end{lemma}
\begin{proof}
    Since $B$ is a division algebra, and $K$ is its center, the subalgebra $\tilde F$ generated by $\iota(F)$ and $K$ must be a field. 

    Let $i\colon D\times \tilde F\to B$ denote multiplication inside $B$, $(x,y)\mapsto \iota(x)y$. The map $i$ is $F$-bilinear, so it factors through a unique $F$-linear map $\tilde\iota\colon D\otimes_F \tilde F\to B$. In fact, $\tilde\iota$ is also $K$-linear, as seen directly from the definition of $i$.

    The property $\tilde\iota|_D = \iota$ holds by definition. We need to show that $\tilde\iota$ is an injective morphism of algebras. Since it sends $1\mapsto 1$, we only need to check it is multiplicative. Let $\alpha,\alpha'\in D$, $f,f'\in F$, and $\kappa,\kappa'\in K$. Then
    \begin{align*}
    \tilde\iota(\alpha\alpha'\otimes ff'\kappa\kappa')
    = \iota(\alpha\alpha'ff')\kappa\kappa'
    &= \iota(\alpha)\iota(f)\kappa \iota(\alpha')\iota(f')\kappa'\\
    &= \tilde\iota(\alpha\otimes\iota(f)\kappa)\cdot
    \tilde\iota(\alpha'\otimes\iota(f')\kappa').
    \end{align*}
    This holds because $F$ is the center of $D$, and $K$ commutes with $\iota(D)$.
    Finally, $\tilde\iota$ is injective because $D\otimes_F \tilde F$ is a simple algebra.
\end{proof}

Next, we move on to computing the necessary capacity. For this, we need to look at orders of Brauer classes. 

\begin{proposition}[Section 13.4 of \cite{pierce82}]\label{proposition:schur-index}
    Let $L/K$ be a finite extension of fields, and let $B$ be a central simple $K$-algebra.
    \begin{enumerate}
        \item If $L$ splits $B$, then $\ord_{K}[B]$ divides $[L:K]$. 
        \item $B$ contains a subfield $L$ that splits $B$, and $\ord_{K}[B]=[L:K]$.
        \item $\ord_{K}[B]$ divides $[L:K]\cdot\ord_{L}[B\otimes_K L]$.
    \end{enumerate}
\end{proposition}

\begin{lemma}\label{lemma:division-algebra-subfield-ord}
	Let $K$ be a field, let $B$ be a central division $K$-algebra, and consider a field $L$ with $K\subseteq L\subset B$. Then
	\[
		\ord_{L}[B\otimes_K L] = \frac{\ord_{K}[B]}{[L:K]}.
	\]
\end{lemma}
\begin{proof}
	By Proposition~\ref{proposition:schur-index}(3) we have
	\[
		\ord_{K}[B] \mid 
		[L:K]\cdot 
		\ord_{L}[B\otimes_K L].
	\]
	To see the \emph{reverse divisibility}, we let $M\subset B$ be a maximal field containing $L$. Then $M$ splits $B$, and in particular it splits $B\otimes_K L$. Therefore by Proposition~\ref{proposition:schur-index}(1) we obtain
	\[
		\ord_{L}[B\otimes_K L] \mid [M:L]
	\]
	and we are done since 
	$[M:L]=\frac{[M:K]}{[L:K]}=
	\frac{\ord_{K}[B]}{[L:K]}.$
\end{proof}

\begin{proposition}\label{proposition:tensor-capacity-KcF}
Let $\ell$ be a prime. Let $\tilde D$ be a division algebra with center $\tilde F$ and Schur index $\ell$, and let $B$ be a division algebra with center $K$. Suppose $K\subseteq \tilde F\subset B$.
    \begin{enumerate}
        \item If $\ell \nmid \frac{\ord_{K}[B]}{[\tilde F:K]}$, then
        \[
            c\left( 
            (B\otimes_K \tilde F) \otimes_{\tilde F} \tilde D^{op}
            \right) 
            =
            [\tilde F:K]
        \]
        \item If $\frac{\ord_{K}[B]}{[\tilde F:K]} = \ell t$ with $\ell\nmid t$, then
        \[
            c\left( 
            (B\otimes_K \tilde F) \otimes_{\tilde F} \tilde D^{op}
            \right) 
            =
            \begin{cases}
                \ell^2[\tilde F:K],\text{ if }
                t[\tilde D] = t[B\otimes_K \tilde F]\text{ in }\Br(\tilde F),\\
                \ell[\tilde F:K],\text{ otherwise.}
            \end{cases}
        \]
        \item If $\ell^2\mid \frac{\ord_{K}[B]}{[\tilde F:K]}$, then
        \[
            c\left( 
            (B\otimes_K \tilde F) \otimes_{\tilde F} \tilde D^{op}
            \right) 
            = \ell[\tilde F:K].
        \]
    \end{enumerate}
\end{proposition}
\begin{proof}
	We let $C=(B\otimes_K \tilde F)\otimes_{\tilde F} \tilde D^{op}$, $c=c(C)$ the capacity of $C$, and $t_C=\ord_{\tilde F}[C]$ its Schur index. The dimension of $C$ is
	\[
		\dim_{\tilde F} C = \dim_{\tilde F} \tilde D\cdot \dim_{\tilde F}(B\otimes_K \tilde F) 
		= \dim_{\tilde F} \tilde D \cdot \dim_K B
		= (\ell \cdot \ord_{K}[B])^2,
	\]
	so by the equality \(\dim_{\tilde F} C = c^2t_C^2\) we obtain
	\begin{equation}\label{eq:capacity}
		c = \frac{\ell \cdot \ord_{K}[B]}{t_C}.
	\end{equation}
	Our task is to compute $t_C$ in the various cases. This is viable since, in $\Br(\tilde F)$, we have $[C]=[B\otimes_K \tilde F]-[\tilde D]$. By assumption, we have $\ord[\tilde D]=\ell$, and Lemma~\ref{lemma:division-algebra-subfield-ord} gives us
	\[
		\ord_{\tilde F}[B\otimes_K \tilde F] = 
		\frac{\ord_{K}[B]}{[\tilde F:K]}.
	\]
	Now we reason by cases.
	\begin{itemize}
		\item If $\frac{\ord_{K}[B]}{[\tilde F:K]}$ is coprime to $\ell$, then $t_C= \ell\frac{\ord_{K}[B]}{[\tilde F:K]}$.
		\item If $\frac{\ord_{K}[B]}{[\tilde F:K]}=\ell t$ with $t$ and $\ell$ coprime, then $t_C = \frac{\ord_{K}[B]}{\ell[\tilde F:K]}$ exactly when $t[B\otimes_K \tilde F]=t[\tilde D]$, and $t_C = \frac{\ord_{K}[B]}{[\tilde F:K]}$ otherwise. 
		\item Finally, if $\ell^2$ divides $\frac{\ord_{K}[B]}{[\tilde F:K]}$, then $t_C = \frac{\ord_{K}[B]}{[\tilde F:K]}$.
	\end{itemize}
	Plugging the computed $t_C$ into Equation \eqref{eq:capacity} yields the stated capacities.
\end{proof}

Finally, we arrive at our main result on division algebras.

\begin{theorem}
\label{theorem:main-theorem-algebras}
    Let $B$ be a central division algebra over a number field $K$. Let $\ell$ be a prime, $F$ a number field, $D$ a division algebra with center $F$ and Schur index $\ell$. Suppose that we are given an embedding of $\QQ$-algebras $\iota\colon  F\hookrightarrow B$. Then, the compositum $\tilde F=\iota(F)K$ in $B$ is a field, and we can extend $\iota$ to an embedding $\iota\colon  D\hookrightarrow B$ if and only if the following conditions hold:
    \begin{enumerate}
        \item $d:= \frac{\ord_{K}{[B]}}{[\tilde F:K]}$ is divisible by $\ell$ exactly once.
        \item $\frac{d}{\ell}[D\otimes_F \tilde F] = \frac{d}{\ell}[B\otimes_{K} \tilde F]$ in $\Br(\tilde F)$.
    \end{enumerate}
    Moreover, when these conditions hold, $\tilde F$ splits neither $D$ nor $B$.
\end{theorem}
\begin{proof}
	Let $\tilde F=\iota(F)K$ and $\tilde D=D\otimes_F \tilde F$.
	By Lemma~\ref{lemma:extension-of-scalars}, we have an embedding $\iota:D\hookrightarrow B$ if and only if we have an embedding $\tilde\iota:\tilde D\hookrightarrow B$. For $\tilde\iota$ to exist, it is necessary that $\tilde D$ be a division algebra, which in particular has Schur index $\ell$. The dimension $\dim_{K}\tilde D=\ell^2[\tilde F:K]$, and so by Theorem~\ref{theorem:yu's-theorem} there is an embedding of $K$-algebras $\tilde D\hookrightarrow B$ if and only if
	\[
		\ell^2[\tilde F:K] \ \Big|\ 
		c\left(
			(B \otimes_{K}\tilde F)
			\otimes_{\tilde F}
			\tilde D^{op}
		\right)=:c.
	\]
	We apply Proposition~\ref{proposition:tensor-capacity-KcF} to see that $\ell^2[\tilde F:K]$ divides $c$ if and only if $d=\frac{\ord_{K}[B]}{[\tilde F:K]}=\ell t$ with $\ell\nmid t$, and
	\[
        t[\tilde F] = t[B \otimes_{K} \tilde F]
    \] in $\Br(\tilde F)$.
	This proves the equivalence statement. When the conditions are satisfied, we have already noted that $\tilde F$ cannot split $D$. To see that $\tilde F$ does not split $B$, we observe that this is not a maximal field of $B$ (alternatively, we may also use the equality of Brauer classes).
\end{proof}

\section{Proof of the main theorem}
\label{section:proof-of-main-theorem}

As hinted in the introduction, our proof of Theorem~\ref{theorem:main} extends the 2-dimensional case. The main idea is found in \cite[Theorem~2.1.4]{schembri19}. For the reader's reference and completeness, we also give it here.

\begin{proposition}\label{proposition:dim-2}
    Let $A/k$ be a simple abelian surface with quaternionic multiplication. Let $\mathfrak p$ be a prime of $k$ over a rational prime $p$. If $A$ has good reduction at $\mathfrak p$ and $\End^0(A)$ does not ramify at $p$, then $A_{\mathfrak p}$ is not simple.
\end{proposition}
\begin{proof}
    Let $\mathfrak p$ be a prime of $k$ of good reduction for $A$ and let $p$ be its residual characteristic. Assume $A_{\mathfrak p}$ is simple, then $\End^0(A_{\mathfrak p})$ is a division algebra. Our goal is to see that $\End^0(A)$ ramifies at $p$.
    
The reduction gives an embedding 
    \[
    \End^0(A)\hookrightarrow \End^0(A_{\mathfrak p}),
    \]
    making $\End^0(A_{\mathfrak p})$ noncommutative. The center of this algebra is the field $\QQ(\pi)$ generated by the Frobenius endomorphism. By \cite[Theorem~2]{tate1966}, $\QQ(\pi)$ strictly contains $\QQ$.
    By the table in \cite[pg. 202]{mumford74} it follows that $\End^0(A_{\mathfrak p})$ is a quaternion algebra over $\QQ(\pi)$, which in turn must be quadratic. Because the center $\QQ$ of $\End^0(A)$ is contained in $\QQ(\pi)$, the algebra $\End^0(A_{\mathfrak p})$ necessarily contains $\End^0(A)\otimes_\QQ \QQ(\pi)$. But now $\dim_{\QQ(\pi)} \End^0(A)\otimes_\QQ \QQ(\pi) = 4 = \dim_{\QQ(\pi)}\End^0(A_{\mathfrak p})$, so in fact we have an isomorphism
    \begin{equation}\label{eq:end-isomorphism}
        \End^0(A_{\mathfrak p}) \simeq \End^0(A)\otimes_\QQ \QQ(\pi).
    \end{equation}
    The field $\QQ(\pi)$ is either real or imaginary quadratic. We may discard the first possibility: by \cite[Proposition~15]{shimura63}, $\End^0(A)$ is an indefinite quaternion algebra, which must remain indefinite after tensoring with $\QQ(\pi)$. However, \cite[Theorem~2]{tate1966} implies $\End^0(A_{\mathfrak p})$ is totally definite whenever $\QQ(\pi)$ is real. Hence $\QQ(\pi)$ is an imaginary quadratic field.

    We end by applying \cite[Theorem~2]{tate1966} once again: when $\QQ(\pi)$ has no real places, $\End^0(A_{\mathfrak p})$ must ramify at some place over $p$. From \eqref{eq:end-isomorphism} it follows that $\End^0(A)$ ramifies at $p$.
\end{proof}

\begin{proof}[Proof of Theorem~\ref{theorem:main}]

Denote by $Z$ the center of $\End^0(A)$ and fix a prime divisor $\ell$ of the Schur index of $\End^0(A)$. By applying Lemma~\ref{lemma:subalgebra-prime-Schur-index} with $E = \End^0(A)$, there is a finite extension $F/Z$, a central division $F$-algebra $D\subseteq \End^0(A)$ with Schur index $\ell$, and an equality of classes
\begin{equation}\label{eq:D-EndA}
    [\End^0(A)\otimes_Z F] = [D]
\end{equation}
in $\Br(F)$. Fix a prime $\mathfrak p$ of $k$ of good reduction for $A$ with residual characteristic $p$. We have the following inclusions of division algebras:
\[
    \xymatrix{
    D_{/F} \ar[r] \ar@/_2pc/[rr]_{\iota} & 
    \End^0(A)_{/Z} \ar[r] &
    \End^0(A_{\mathfrak p})_{/\QQ(\pi)}.
    }
\]
We focus on the embedding $\iota\colon D\hookrightarrow \End^0(A_{\mathfrak p})$. Suppose that $A_{\mathfrak p}$ is simple: in that case, $\End^0(A_{\mathfrak p})$ is a division algebra, and we want to see that $\End^0(A)$ ramifies at some prime over $p$. We may apply Theorem~\ref{theorem:main-theorem-algebras} with $K=\QQ(\pi)$ and $B=\End^0(A_{\mathfrak p})$. We denote by $F(\pi)$ the compositum of $\iota(F)$ and $\QQ(\pi)$ in $\End^0(A_{\mathfrak p})$. Then, the existence of the embedding $\iota$ implies that
\[
	d := \frac{\ord_{\QQ(\pi)}[\End^0(A_{\mathfrak p})]}{[F(\pi):\QQ(\pi)]}
\]
is divisible by $\ell$ exactly once, and
\begin{equation}\label{eq:D-EndAp}
	\frac{d}{\ell}[D\otimes_F F(\pi)] = \frac{d}{\ell}[\End^0(A_{\mathfrak p})\otimes_{\QQ(\pi)}F(\pi)]
\end{equation} 
in $\Br(F(\pi))$.
With this $d$, we see that the Brauer class $\frac{d}{\ell}[D\otimes_F F(\pi)]$ is not trivial. Indeed, $F(\pi)$ does not split $D$, so $D\otimes_F F(\pi)$ has Schur index $\ell$, while $\frac{d}{\ell}$ is an integer coprime with $\ell$. Combining Equations \eqref{eq:D-EndA} and \eqref{eq:D-EndAp} we obtain an equality of non-trivial classes in $\Br(F(\pi))$,
\[
\frac{d}{\ell}[\End^0(A)\otimes_Z F(\pi)] = \frac{d}{\ell}[\End^0(A_{\mathfrak p})\otimes_{\QQ(\pi)} F(\pi)].
\]
By Proposition~\ref{proposition:dim-2}, we may assume that $\dim A>2$. Therefore, $\QQ(\pi)$ is a CM field by \cite[pg.~ 97]{tate71}. By \cite[Theorem~2]{tate1966}, the algebra $\End^0(A_{\mathfrak p})$ ramifies only at primes of $\QQ(\pi)$ over $p$. It follows that $\End^0(A)$ must ramify at some prime of $Z$ over $p$, this proves our theorem.
\end{proof}

\nocite{achter12}

\bibliographystyle{alpha}
\bibliography{references}
\end{document}